\documentclass[10pt]{amsart}
\usepackage{amsmath, amsthm, amsfonts, amssymb, graphicx}
\usepackage[all]{xy}

\usepackage{color}
\usepackage{fullpage}

\newtheorem{thm}{Theorem}[section]
\newtheorem{theorem}[thm]{Theorem}
\newtheorem{cor}[thm]{Corollary}
\newtheorem{lemma}[thm]{Lemma}

\theoremstyle{remark}

\theoremstyle{definition}
\newtheorem{definition}[thm]{Definition}

\newcommand{\anc}{\operatorname{anc}}
\newcommand{\des}{\operatorname{des}}

\newcommand\peg[3]{\xymatrix{#1 \ar@/^2mm/[r]_{#2}& #3}}

\begin{document}

%
%

\title{Pegging Numbers for Various Tree Graphs}
\author{Ariel Levavi}
\date{}
\maketitle

\begin{abstract}
	In the game of pegging, each vertex of a graph is considered a hole into which a peg can be placed. 
	A pegging move is performed by jumping one peg over another peg, and then removing the peg that has been jumped over from the graph. 
	We define the pegging number as the smallest number of pegs needed to reach all the vertices in a graph no matter what the distribution. 
	Similarly, the optimal-pegging number of a graph is defined as the smallest distribution of pegs for which all the vertices in the graph can be reached. 
	We obtain tight bounds on the pegging numbers and optimal-pegging numbers of complete binary trees and compute the optimal-pegging numbers of complete infinitary trees. 
	As a result of these computations, we  deduce that there is a tree whose optimal-pegging number is strictly increased by removing a leaf.  
	We also compute the optimal-pegging number of caterpillar graphs and the tightest upper bound on the optimal-pegging numbers of lobster graphs.
\end{abstract}


\section{Introduction}\label{intro}

	One of the better known peg solitaire games is described by Berlekamp, Conway, and Guy in their book, \textit{Winning Ways for Your Mathematical Plays} \cite{gridGame}. 
	In this game we are presented with an infinite square grid on a Cartesian plane. 
	At each intersection there is a hole in which a peg can be placed, and all the holes in the lower half-plane are filled. 
	We can move pegs on the grid in the following manner: if we have two adjacent pegs, one of which is adjacent to an empty hole, then the peg further from the hole can jump over the peg next to the hole and fill it. 
	The peg that was jumped over is then removed from the grid. 
	As it turns out, the farthest up the grid we can get using only legal pegging moves is a distance of 4 holes.

	A graph version of this game, where in place of a grid we have any graph where each vertex represents a hole that can be filled by a peg, was studied by Helleloid, Khalid, Matchett-Wood, and Moulton in \cite{Moulton} and later by Wood in \cite{Wood}. 
	This game is a variation of the graph game \textit{pebbling} \cite{pebbling}, where each move removes 2 pebbles from a vertex and places one of the pebbles on an adjacent vertex. 
	A pegging move on a graph consists of removing two pegs from adjacent vertices and placing one peg on a unfilled vertex adjacent to one of the first two holes. Given some initial distribution of pegs on a graph, we say we can \textit{reach} a vertex if after a (possibly empty) series of pegging moves we can cover it with a peg. 
	Research on pegging mainly involves finding the \textit{optimal-pegging number}, the size of the smallest distribution which can reach all the vertices in a graph, and finding the \textit{pegging number}, the smallest number of pegs needed to reach all the vertices in a graph no matter what the distribution. 
	These terms are modeled after the definitions of \textit{optimal-pebbling numbers} and \textit{pebbling numbers}, respectively. 
	In their paper, Helleloid et al. study the pegging and optimal-pegging numbers of various types of graphs including paths, cycles, hypercubes, and complete graphs. 
	They also explore basic pegging properties of graphs with diameters of at most 3. 
	Wood examines various graph products and the effects of small variations of distributions of pegs on their reach.

	This paper explores the pegging properties of various types of tree graphs. 
	In Section \ref{background}, we provide background and terminology pertaining to the pegging of general graphs. 
	In Section \ref{genTrees} we present pegging properties of general trees. 
	In Section \ref{binaryTrees}, we find tight bounds on the pegging and optimal-pegging numbers of the complete binary tree.
	In Section \ref{infTrees}, we define the concept of an infinitary tree, and find the optimal-pegging number of this tree, which follows the Fibonacci sequence. 
	As a result of the computations of the optimal-pegging numbers for these two classes of trees, we deduce that there is a tree whose optimal-pegging number is strictly  increased by removing a leaf. 
	In Section \ref{caterpillars}, we compute the optimal-pegging numbers of caterpillar graphs, as well as the lowest upper bound on the optimal-pegging numbers of lobster graphs. 
	Finally, in Section \ref{conclusion}, we discuss further research possibilities.

\section{Background and Terminology}\label{background}

	Although this paper only discusses trees, we will describe the basics of pegging for general graphs. 
	Let $G$ be a graph, and imagine that several of its vertices are filled with one peg each. 
	Then we consider this set of filled vertices $D$ to be a \textit{distribution of pegs} on $G$.
	Given distributions $D$ and $D'$ on $G$, we say that $m = \peg{u}{v}{w}$ is a \textit{pegging move} from $D$ to $D'$ if there exist distinct vertices $u, v \in D$ and $w \notin D$ with both $u$ and $w$ adjacent to $v$ and with $D' = (D \backslash \{u, v\}) \cup \{w\}$. 
	Informally, $m$ describes the act of the peg on $u$ jumping over the peg on $v$ and landing on $w$, as shown in Figure 1 below. 
	
\begin{figure}[h]
	\centering
		\includegraphics[scale=0.5]{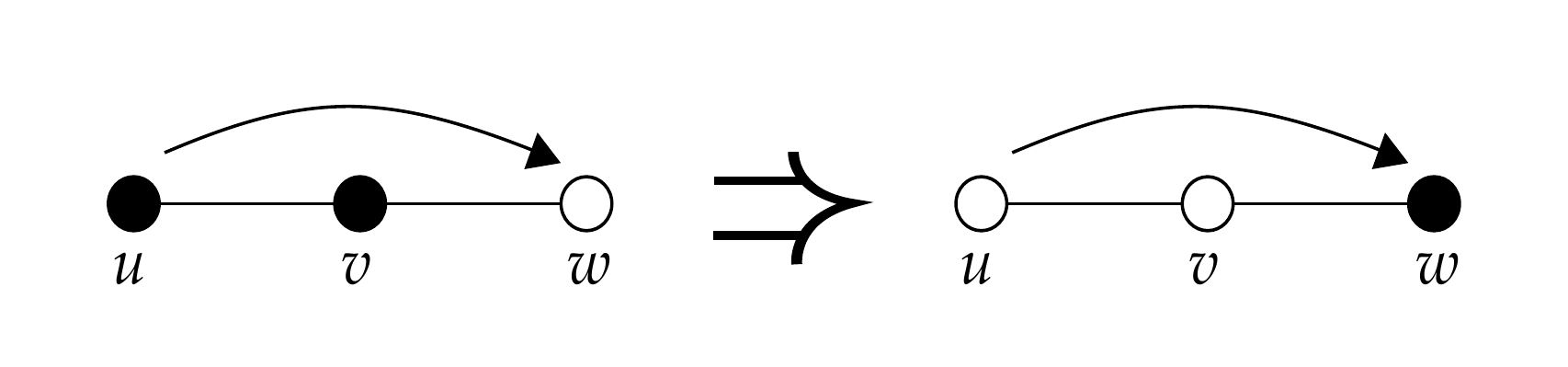}
	\caption{Illustration of the pegging move $m$.}
\end{figure}

	In this case we write $D' = m(D)$. 
	If $D'$ is achieved after performing a finite sequence of pegging moves $M=(m_1,m_2,\dots, m_l)$ on $D$, we write $D'=M(D)$.
	We denote the sequence of the first $i$ moves in $M$ by $M_i$ for $1\leq i\leq l$. 
	We say that $v$ is \textit{reachable} (or can be \textit{pegged}) from a distribution $D$ if there exists a finite sequence of pegging moves $M$ with $v \in M(D)$. 
	The \textit{reach} $R_G(D)$ of a distribution $D$ on $G$ is the set of all vertices reachable from $D$. 
	We say that $G$ can be \textit{pegged} by $D$ if $R_G(D) = V(G)$. 
	If the graph $G$ is clear from the context, we will denote the reach of $D$ as simply $R(D)$.

We will now introduce our main definitions.

\begin{definition}\label{defPegging}
	The \textit{pegging number} $P(G)$ of a graph $G$ is the smallest natural number such that $G$ is peggable by every distribution of size $P(G)$. 
	The \textit{optimal-pegging number} $p(G)$ of $G$ is the minimum size of all distributions on $G$ that can peg $G$.
\end{definition}

	To illustrate these concepts, consider $S_n$, the star graph on $n$ vertices (see Figure 2). 
	Suppose our distribution is $D=\{v\mid v\mbox{ is a leaf}\}$. 
	Since no two vertices in $D$ are adjacent, it is impossible to perform any pegging moves. 
	Therefore, the center vertex cannot be reached. 
	So there exists a distribution $D$ of $n-1$ pegs on $S_n$ such that $R(D)\neq V(S_n)$, which implies that $P(S_n)=n$.

	Now suppose $D$ consists of only two pegs with one placed on the center vertex and one placed on one of the leaves. 
	Since all of the leaves in  $S_n$ are adjacent to the center vertex, they can all be reached. 
	So $p(S_n)\leq 2$. 
	Since we can't perform any pegging moves if $|D|=1$, $p(S_n)=2$.

\begin{figure}[h]
	\centering
		\includegraphics[scale=0.4]{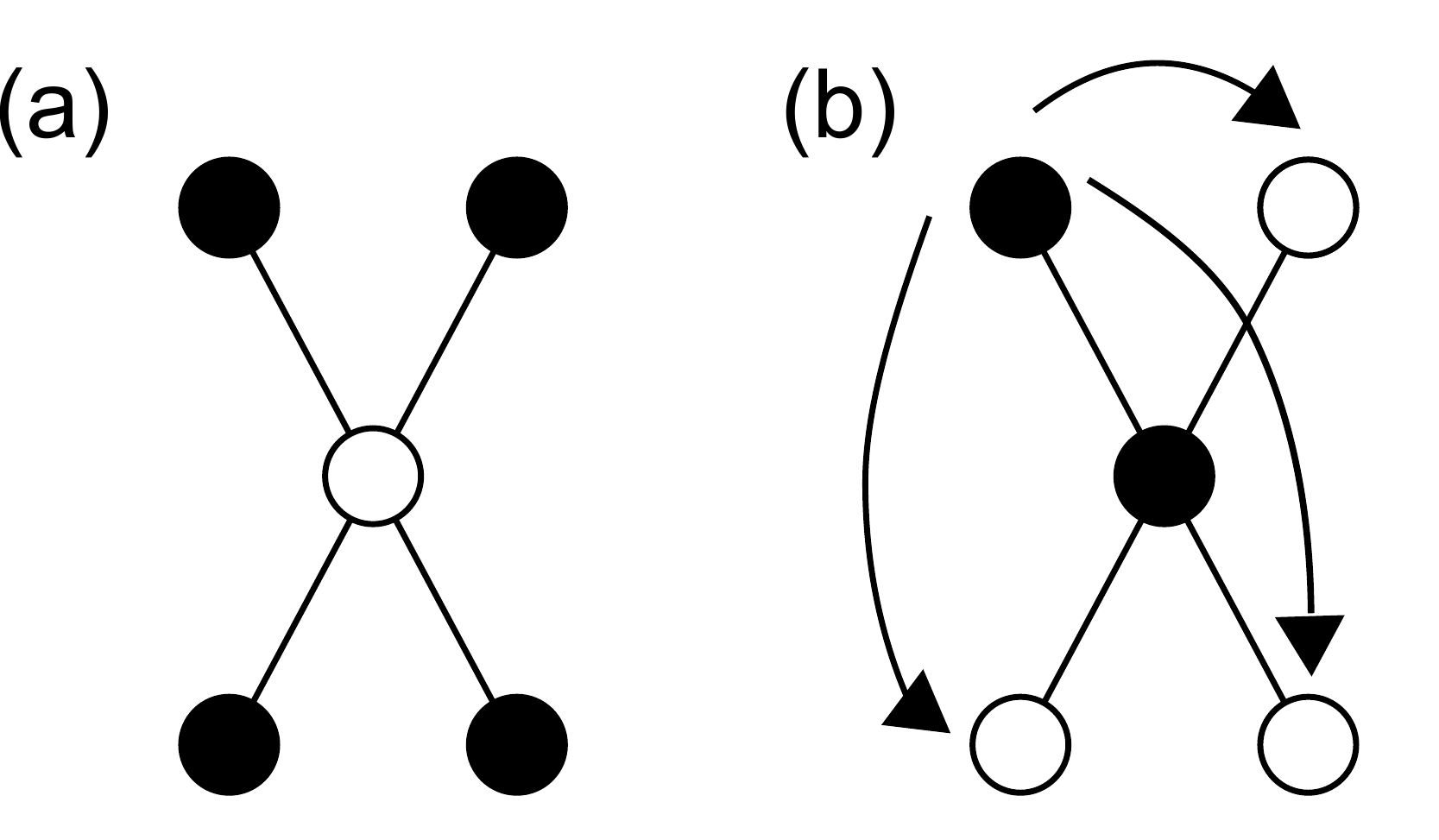}
	\caption{Illustration of the pegging and optimal pegging numbers of $S_5$. Graph (a) shows that four pegs do not necessarily reach all the vertices of $S_5$. However, the two well placed pegs on graph (b) achieves the goal. }
\end{figure}

	The next set of definitions is useful for calculating the pegging and the optimal-pegging numbers of a graph.

\begin{definition}\label{defWeight}
	Let $t\in V(G)$. 
	The \textit{weight of a vertex} $v$ \textit{with respect to} $t$ is $\omega^{d(v,t)}$, where $\omega = (\sqrt{5}-1)/2$ and $d(v,t)$ is the distance between vertices $v$ and $t$. 
	The \textit{weight of a distribution} $D$ \textit{with respect to} $t$ is 
	
$$
w_t(D) = \sum_{v\in D} \omega^{d(v,t)},
$$

	The \textit{summed weight of a vertex} $v$ \textit{with respect to} $L\subseteq V(G)$ is

$$
w_{L}(v) = \sum_{t\in L} w_t(v).
$$

	Similarly for $D\subseteq V(G)$, $w_{L}(D) = \sum_{t\in L} w_t(D)$.
\end{definition}

	Notice that $\omega$ defined in this way is a positive solution of the $x^2+x=1$. 
	Weights satisfy the first of the following two monotonicity lemmas. 
	The second monotonicity lemma establishes that removing pegs from a distribution cannot help us reach vertices that we otherwise could not reach with their presence.

\begin{lemma}[Monotonicity of Weight]\label{monoWeight}
	Let $D$ be a distribution of pegs on a graph $G$, and let $D'$ be the distribution achieved by performing the finite sequence of pegging moves $M$. 
	Then $w_t(D')\leq w_t(D)$ for all $t\in V(G)$.
\end{lemma}

\begin{lemma}[Monotonicity of Reach]\label{monoReach}
	Let $D'\subseteq D$ be two distributions on a graph $G$. 
	Then we have $R(D')\subseteq R(D)$.
\end{lemma}

	See \cite{Moulton} for proofs.

	Combining Lemma \ref{monoWeight} with the fact that if $t\in D$, $w_t(D)\geq1$, we get the following theorem:

\begin{theorem}\label{reachThm}
	If $D\subseteq V(G)$ and $t\in R(D)$, then $w_t(D)\geq 1$.
\end{theorem}

	The contrapositive of this theorem is extremely useful for calculating lower bounds for pegging and optimal-pegging numbers. 
	Here is another variation of this theorem.

\begin{cor}\label{reachSummedCor}
	If $L\subseteq R_G(D)$, then $w_{L}(D)\geq |L|$.
\end{cor}

\begin{proof}
	Suppose $L\subseteq R_G(D)$. Then 
$\displaystyle
w_{L}(D)	=	\sum_{t\in L} w_t(D)
\geq 	\sum_{t\in L}1 
= |L|.$

\end{proof}

	Corollary \ref{reachSummedCor} shows that we can obtain a lower bound for the optimal-pegging number of any graph $G$ by choosing $L\subseteq V(G)$ and taking the partial sum $S_k$ of the $k$ largest summed weights of the vertices in $V(G)$ with respect to $L$. 
	We sum until we achieve some $S_k$ where $S_{k-1}<|L|\leq S_k$.

	One interesting variation of pegging involves stacking more than one peg on a single vertex. 
	We define a \textit{multi-distribution of pegs} $D$ on a graph $G$ as a multiset of the elements of $V(G)$. 
	Informally, $D$ is a distribution of pegs on $G$ where each vertex can be covered with more than one peg.
	To distinguish between a multi-distribution and our original definition of a distribution, we will call distributions that only allow one peg per vertex \textit{proper distributions}. 
	A \textit{stacking move} $m=\peg{u}{v}{w}$ from a multi-distribution $D$ is a pegging move without the constraint that $w\notin D$. 
	Notice that all pegging moves are stacking moves, but the converse is not true. 
	Also notice that the definitions of pegging numbers and optimal-pegging numbers require that the initial distributions of pegs be proper. 
	We say that $D$ \textit{supersedes} $D'$ if the presence of peg $p\in D'$ implies that $p\in D$. 
	The set of vertices that are reachable by $D$ via stacking moves is denoted $R_s(D)$.

	Let $D_0$ be a multi-distribution of pegs and let $M=(m_1,m_2,\ldots,m_l)$ be a sequence of stacking moves with $D_i=M_i(D_0)$. 
	Fix a peg $p\in D_l$. Informally, we define the \textit{ancestors} of $p$, denoted $\anc(D_0,M,p)$, as the set of pegs in $D_0$ that contribute to the placement of $p\in D_l$. 
	Formally, we define the ancestor function by induction:

\begin{enumerate}
\item If $M=\emptyset$, then $\anc(D_0,\emptyset, p) = \{p\}$.
\item If $m_i=\peg{u}{v}{p}$, then $\anc(D_0,M_i,p) = \anc(D_0,M_{i-1},u) \cup \anc(D_0,M_{i-1},v)$.
\item If the specific occurrence of $p\in D_i$ is also in $D_{i-1}$, then $\anc(D_0,M_i,p) = \anc(D_0,M_{i-1},p).$
\end{enumerate}

\begin{lemma}\label{ancestorLem}
	Let $D_0$ be a distribution of pegs and let $M=(m_1,\ldots,m_l)$ be a sequence of moves to obtain some distribution $D_l$ containing pegs $x\neq y$. 
	Then $\anc(D_0,M,x)\cap \anc(D_0,M,y)=\emptyset$.
\end{lemma}

\begin{proof}
	Let $M_i=(m_1,\cdots,m_i)$ be the first $i$ moves in $M$ where $1\leq i\leq l$ and denote the distribution we achieve after $M_i$ as $D_i$. 
	Informally, we define the function \textit{des} as the peg in $D_l$ that is contributed to by some peg in $D_0$. 
	Formally, we define $\des(p)=d(D_0,M,p)$ where we define function $d$ by induction:

\begin{enumerate}
\item If peg $p\in D_l$, then $d(D_l,M,p)=p$.
\item If $m_i=\peg{u}{v}{p}$, then $d(D_{i-1},M,u)=d(D_{i-1},M,v)=\des(D_i,M,p)$.
\item If $p\in D_{i-1}\cap D_i$, then $d(D_{i-1},M,p)=d(D_i,M,p)$.
\end{enumerate}

	Since preimages of distinct vertices under \textit{des} are disjoint, the sets $\anc(D_0,M,x)$ and $\anc(D_0,M,y)$ are disjoint.
\end{proof}

	The following lemma provides a property about pegging graph that we will use frequently in the remainder of this paper.

\begin{lemma}\label{pegRoutesLemma}
	Let $u,v\in V(G)$, and let $D$ be a distribution of pegs on $G$. 
	Suppose $M$ is a sequence of pegging moves such that $u\in\anc(D,M,v)$.
	Then for a vertex $w\in D$ there exists a path $P=\{w,v_1,v_2,\dots,v_k,v\}$ such that $u\in P$.
\end{lemma}

\begin{proof}
	Suppose $u\in D$. 
	If we use $u$ to reach $v$ then there must exist a path between the two vertices, and the above statement is trivially true.
	Suppose $u\notin D$. In order for a peg to reach $u$, there must be a vertex $w\in D$ that is an ancestor of $u$, which implies that there is a path between $w$ and $u$.
	So by transitivity, there is a path from $w$ to $v$ containing $u$.
\end{proof}

	The following lemma shows that stacking pegs does not extend the reach of a proper distribution.

\begin{lemma}[Stacking Lemma]\label{stacking}
	Let $D$ be a proper distribution, and let $M$ be a sequence of stacking moves leading to a multi-distribution $M(D)$. 
	Then there exists a sequence of pegging moves $M'$ resulting from a reordered subsequence of $M$ such that the proper distribution $M'(D)$ supersedes $M(D)$. 
	Moreover, $R_s(D)=R(D)$.
\end{lemma}

	See \cite{Moulton} for proof.

	Another graph game that is similar to pegging is \textit{pebbling}. 
	A \textit{pebbling move} on a multi-distribution $D$ is one where we remove two pegs from one vertex and place one on an adjacent vertex. 
	A \textit{peggling move} is either a stacking move or a pebbling move. 
	The set of vertices that are reachable by $D$ via peggling moves is denoted $R_b(D)$.

	The next lemma found in \cite{Moulton} shows that pebbling does not increase the reach of a proper distribution either.

\begin{lemma}[Peggling Lemma]\label{pegglingLem}
	Let $D$ be a proper distribution and let $M$ be a sequence of peggling moves leading to a multi-distribution $M(D)$. 
	Then there exists a sequence of pegging moves $M'$ such that the proper distribution $M'(D)$ supersedes $M(D)$. 
	Moreover, $R_b(D)=R(D)$.
\end{lemma}

	Since these peg moving methods do not change the reach, they are useful techniques that will come in handy later on.

\section{General Trees}\label{genTrees}

	In the sections to come, we will introduce various types of trees and compute their pegging numbers and optimal-pegging numbers. 
	The goal of this section is to present two results that are true for all trees. 
	The first result is due to M. Wood \cite{Wood}.

\begin{lemma}\label{Melanie}
	Let $T$ be a tree, and let $t$ be a vertex in $T$ that is reachable from some distribution $D$. 
	Then $t$ is reachable from $D$ using only pegging moves toward $t$, i.e. for each move $m=\peg{u}{v}{w},$ $d(w,t)<d(u,t)$.
\end{lemma}

	The idea behind the proof is that since $T$ is acyclic there exists only one path between $t$ and $v\in D$. 
	Therefore if we move $v$ \textit{away} from $t$ we will eventually have to move toward $t$ on the same path, and so such a move would unnecessary. 

	The second result in this section pertains to the pegging numbers of subtrees.

\begin{lemma}\label{subtreePeg}
	Let $T$ be a tree and $T'$ a subtree of $T$. 
	Then $P(T')\leq P(T)$.
\end{lemma}

\begin{proof}
	Let $k=P(T)$. 
	If $k> |V(T')|$, then trivially $P(T')<P(T)$. 
	If $k\leq |V(T')|$, then let $D$ be a distribution of size $k$ on $T'$. 
	By definition of $k$, $R_T(D)=V(T)\supseteq V(T')$. 
	Suppose there exist vertices $u\in V(T)\backslash V(T')$ and $v\in V(T')$ such that we cannot reach $v$ without first reaching $u$. 
	By Lemma \ref{pegRoutesLemma} this can only occur if there exists a path between $v$ and some vertex $w\in D$ that contains $u$. 
	Since $T$ is acyclic, there exists at most one path from $w$ to $v$, which is contained in $T'$.
	So $u$ is not contained is this path, which implies that $R_{T'}(D) = V(T)$.  
	Hence, $P(T')\leq P(T)$.
\end{proof}

	Note that this property of pegging numbers does not hold for general graphs. 
	Nor does the same conclusion hold for optimal-pegging numbers of graphs, trees or otherwise.

\section{Complete Binary Trees}\label{binaryTrees}

	In this section we will find bounds for the pegging number and optimal-pegging number of the complete binary tree. 
	We will denote the complete binary tree of height $h$ as $T_h$. 
	A vertex $v$ is located at \textit{level} $l$ of $T_h$ if $d(v,r) = l$, where $r$ is the root of the tree. 
	By this terminology, the root $r$ is the vertex located at level 0. 
	We will use the notation $v:l$ to signify that vertex $v$ lies at level $l$.

	The following lemma is pivotal for calculating a lower bound for the pegging number of these trees.

\begin{lemma}\label{rootWeightLem}
If $r$ denotes the root of $T_h$, then

$$
w_r(V(T_h)) = \frac{(2 \omega)^{h+1}-1}{2 \omega-1}.
$$

\end{lemma}

\begin{proof}
	A vertex at level $l$ is a distance of $l$ away from the root. 
	Furthermore, the number of vertices at level $l$ is $2^l$. 
	So
	
\begin{equation*}
w_r(V(T_h)) = \sum\limits_{v\in V(T_h)}w_r(v)
=	\sum\limits_{l=0}^h\sum\limits_{v:l}\omega^l
= \sum\limits_{l=0}^h(2\omega)^l
= \frac{(2\omega)^{h+1} -1}{2\omega -1}.
\end{equation*}

\end{proof}

	Now we will present our bound.

\begin{theorem}\label{pegBinaryThm}
	Let $T_h$ be a complete binary tree of height $h$. 
	Then for sufficiently large $h$ we have
\begin{equation*}
P(T_h)\geq |V(T_h)|-172.
\end{equation*}
\end{theorem}

\begin{proof}
	By Lemma \ref{rootWeightLem}, each subtree of $T_h$ of height $k$, denoted $\widehat{T}_k$, rooted at a vertex $v\in V(T_h)$ has weight $w_v(V(\widehat{T}_k)) = \frac{(2 \omega)^{k+1}-1}{2 \omega-1}$. 
	For each leaf $t\in V(T_h)$ there exists a subtree of height $k$ whose root is located a distance of $k+2$ away from it. We find this tree by following the path of length $k+1$ from $t$ toward the root. 
	Once we are a distance of $k+1$ away from $t$ at some vertex $u$, we go to the child of $u$ that we have not traversed. 
	This child is at level $h-k$ and therefore is the root of a subtree of height $h-k$. 
	Since this subtree is a distance $k+2$ from $t$, its weight with respect to $t$ is

\begin{equation*}
w_t(V(\widehat{T}_k)) = \omega^{k+2}\frac{(2 \omega)^{h+1}-1}{2 \omega-1}.
\end{equation*}

	Let us sum up the weights with respect to $t$ of every subtree of height $8\leq k\leq h-1$ that is a distance of $k+2$ away from $t$. 
	Since $T_h$ is a binary tree, there exists only one such tree of each height $k$ that fits these criteria and all of these trees are disjoint. 
	In addition, let us add the weights with respect to $t$ of each vertex along the path between $t$ and the root which is a distance of $k+1$ from $t$ for $8\leq k\leq h-1$. 
	In other words, we are adding the weights with respect to $t$ of every vertex in $V(T_h)$ and subtracting the subtree of height 7 that contains $t$ as a leaf. The weight of this distribution $D$ is

\begin{align*}
w_t(D) 
= &\sum\limits_{k=0}^{h-1}\left(\omega^{k+2}\frac{(2 \omega)^{k+1}-1}{2 \omega-1}+ \omega^{k+1}	\right)\\
&-\sum\limits_{k=0}^{7}\left(\omega^{k+2}\frac{(2 \omega)^{k+1}-1}{2 \omega-1}+ \omega^{k+1}	\right)\\
\leq &\sum\limits_{k=0}^{h-1}\omega^{k+2}\frac{(2 \omega)^{k+1}-1}{2 \omega-1}+ \sum\limits_{k=0}^{h-1}\omega^{k+1} - 4.9271	\\
= &\frac{1}{2\omega-1}\sum\limits_{k=0}^{h-1}\left(2^{k+1}\omega^{2k+3}- \omega^{k+2}\right)+ \omega\sum\limits_{k=0}^{h-1}\omega^{k} - 4.9271\\
=& \frac{2\omega^3}{2\omega-1}\sum\limits_{k=0}^{h-1}(2 \omega^2)^{k} + \omega\left(1-\frac{\omega}{2\omega -1}	\right)\sum\limits_{k=0}^{h-1}\omega^{k} - 4.9271\\
= &\frac{2\omega^3(2^h\omega^{2h}-1)}{(2\omega-1)(2\omega^2-1)} + \frac{\omega}{\omega-1}\left(1-\frac{\omega}{2\omega -1}	\right)(\omega^h-1) - 4.9271\\
\end{align*}

	As $h\rightarrow\infty$,

\begin{align*}
\lim_{h\rightarrow\infty}w_t(D)
&< \frac{2\omega^3}{(2\omega-1)(1-2\omega^2)} + \frac{\omega}{1-\omega}\left(1-\frac{\omega}{2\omega -1}	\right) - 4.9271\\
&< 8.47213 -2.61803- 4.9271\\
&=0.927<1.
\end{align*}

	So by Theorem \ref{reachThm}, we have $t\notin R_{T_h}(D)$, which implies that $P(T_h)>|D|$. 
	Since $|D|= |V(T_h)|-|V(T_8)|$, for sufficiently large $h$ we have
	
\begin{equation*}
P(T_h)
 \geq |D|+1 
 = |V(T_h)|-|V(\widehat{T}_7)|+1
 = |V(T_h)|-254.
\end{equation*}

\begin{figure}[b]
	\centering
		\includegraphics[scale=0.55]{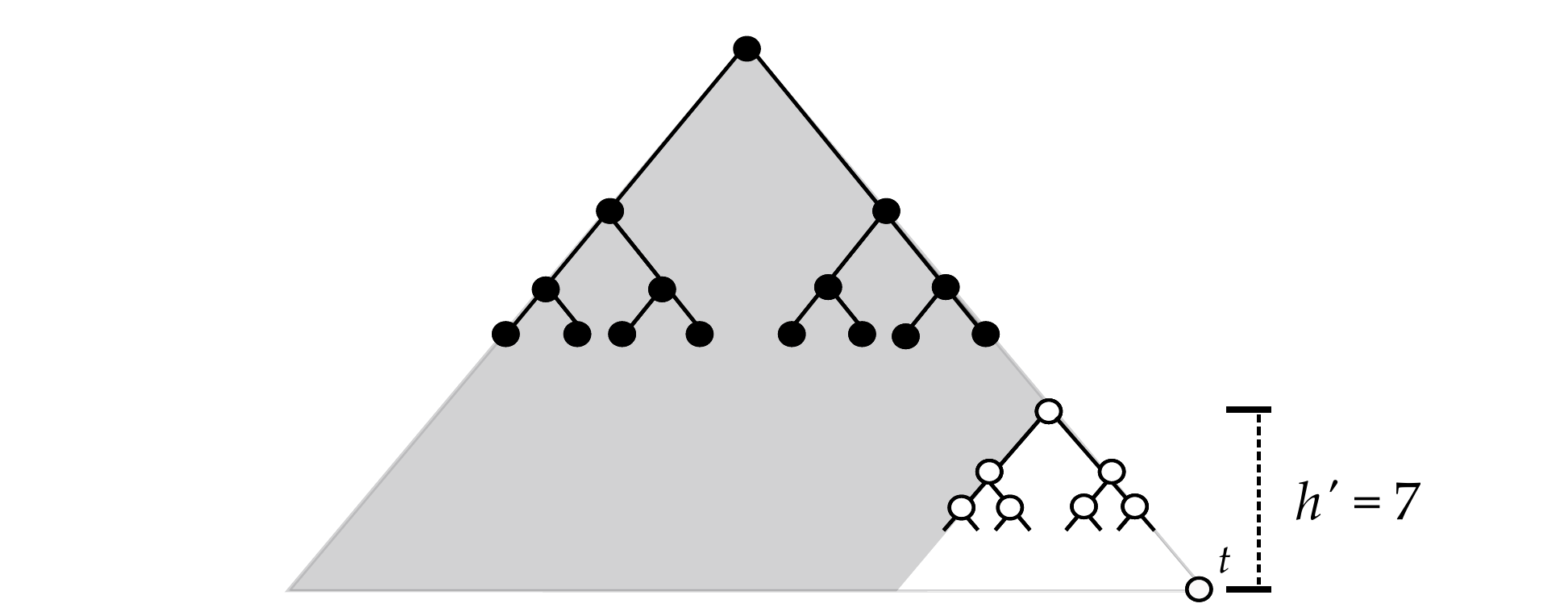}
	\caption{The binary tree $T_h$ with distribution $D$ such that every vertex is covered with a peg except for the vertices in the subtree $\widehat{T}_7$ on the bottom right. By Theorem 4.2 we see that, given this distribution, we cannot reach vertex $t$. Theorem 4.2 further refines this result.}
\end{figure}
	
	Now let us refine our results. 
	When $h\geq 11$, we have exactly 1 pegged vertex a distance of 8 from $t$, 2 pegged vertices a distance of 9 from $t$, 4 pegged vertices a distance of 10 from $t$, and 8 pegged vertices a distance of 11 from $t$ in the distribution $D$.
	In $V(T_h)\backslash D$ there are 64 vertices of distance 12 from $t$, 32 vertices of distance 13 from $t$, and 48 vertices of distance 14 from $t$. 
	If we remove all pegs from vertices of $D$ of distances 8, 9, 10, and 11 from $t$ and add pegs to all the vertices in $V(T_h)\backslash D$ of distances 13 and 14 from $t$ and 1 of the 48 vertices of distance 12 from $t$, we get a distribution $D'$ such that

\begin{align*}
\lim\limits_{h\rightarrow\infty}w_t(D')
&=  \lim\limits_{h\rightarrow\infty}w_t(D) + 64 \omega^{14} + 32\omega^{13} + \omega^{12} - 8\omega^{11} - 4\omega^{10} - 2\omega^{9}-\omega^8\\
&=  0.99975
<  1.
\end{align*}

	So by Theorem \ref{reachThm}, $t\notin R_{T_h}(D')$, which implies that $P(T_h)>|D'|$. 
	Since $|D'|= |V(T_h)|-|V(T_7)|+82$, for sufficiently large $h$ we have
	
\begin{equation*}
P(T_h)
 \geq |D'|+1 
= |V(T_h)|-|V(T_7)|+83
= |V(T_h)|-172.
\end{equation*}
\end{proof}

\begin{cor}\label{pegLimit}
\begin{equation*}
\lim\limits_{h\rightarrow\infty}\frac{P(T_h)}{|V(T_h)|} = 1.
\end{equation*}
\end{cor}

\begin{cor}\label{pegBinaryCor}
	Let $T_h$ be a complete binary tree of height $h$. 
	Then

$$
P(T_h)\geq |V(T_h)|-172
$$
for all $h\in\Bbb N$.
\end{cor}

\begin{proof}
	If $|V(T_h)|\leq 172$ then this statement is trivially true. 
	Therefore, assume $|V(T_h)|> 172$. 
	Then $T_h$ is a subtree of $T$, a complete binary tree that we consider to have a ``sufficiently large height.''
	Let $D$ be a distribution of pegs on $T$ and $D'$ a distribution of pegs on $T_h$.
	We use the construction of $D$ described in Theorem \ref{pegBinaryThm} so that there exists at least one vertex $t\in V(T)$ that we cannot reach.
	If $h\geq 11$, then we can construct $D'$ in such a way that $T_h$ contains the tree of height 7 that was originally unpegged.
	In accordance of our refined results, at least 157 of the pertinent 172 vertices are unpegged.
	So $D'\subseteq D$, and by Monotonicity of Reach, $R(D')\subseteq R(D)$.
	Hence, $t\notin R(D')$, and $P(T_h)\geq |V(T_h)|-172$.
\end{proof}

	The next two lemmas are necessary for calculating bounds for the optimal-pegging number for $T_h$.

\begin{lemma}\label{summedWeightLem}
	Let $T_h$ be a complete binary tree, and let $L\subseteq V(T_h)$ be the set of leaves of $T_h$. 	
	Then for any vertex where $v\in V(T_h)$ is at level $l$, we have

$$
w_{L}(v) = \frac{(2\omega)^{h-l}}{\omega^2}\left(1-\omega(2\omega^2)^l	\right).
$$
\end{lemma}

\begin{proof} We prove this by strong induction on $l$.

	\textit{Base Case:} $l=0$. 
	The root $r$ is the only vertex in $V(T_h)$ that is located at level 0. 
	Since $T_h$ is complete, all of the leaves in $T_h$ are located at level $h$, and so for all leaves $t$, $d(r,t)=h$. 
	Hence,

\begin{align*}
w_{L}(r) &=	\sum_{t\in L} w_t(r)
= \sum_{t\in L} \omega^{d(r,t)}	
=	\sum_{t\in L} \omega^{h}\\
&= 2^h\cdot \omega^h 
= (2\omega)^{h-0}\left(\frac{1-\omega}{\omega^2}\right)
= \frac{(2\omega)^{h-l}}{\omega^2}\left(1-\omega(2\omega^2)^0	\right).
\end{align*}

\textit{Inductive Hypothesis:} Assume that for all vertices at level $k$ that is at most $l-1$, we have 
$$
w_{L}(v) = \frac{(2\omega)^{h-k}}{\omega^2}\left(1-\omega(2\omega^2)^k	\right).
$$

	\textit{Inductive Step:} Consider the vertex $v$ located at level $l$ of $T_h$. 
	The subtree rooted at $v$ has a height of $h-l$, and so by the base case the summed weight of its leaves is $(2\omega)^{h-l}$. 
	For every other leaf in $L$, the summed weight of $v$ is the summed weight of the vertex $u$ adjacent to $v$ at level $l-1$ multiplied by $\omega$, since $v$ is a distance of 1 away from $u$. 
	However, when we add $\omega\times w_{L}(u)$ we have double counted the leaves of the subtree rooted at $v$. 
	We must therefore subtract $\omega^2-w_L(v)$.
	So our summed weight for $v$ with respect to $L$ is:

\begin{align*}
w_{L}(v) &=
 (2\omega)^{h-l}(1-\omega^2) + \omega\cdot w_{L}(u)\\
&=(2\omega)^{h-l}(1-\omega^2)+\frac{(2\omega)^{h-l+1}}{\omega}\left(1-\omega(2\omega^2)^{l-1}	\right)\\
&=(2\omega)^{h-l}\left[1-\omega^2+2\left(1-\omega(2\omega^2)^{l-1}	\right)	\right]\\
&=\frac{(2\omega)^{h-l}}{\omega^2}\left(\omega^2-\omega^4+2\omega^2-2\omega^3(2\omega^2)^{l-1}	\right)\\
&=\frac{(2\omega)^{h-l}}{\omega^2}\left(3\omega^2-\omega^4-\omega(2\omega^2)^{l}	\right)\\
&=\frac{(2\omega)^{h-l}}{\omega^2}\left(1-\omega(2\omega^2)^{l}		\right).
\end{align*}
\end{proof}

	We proceed using the technique following Corollary \ref{reachSummedCor} to find a lower bound for the optimal-pegging number. 
	Our first step is to rank each level of the tree according to the weights of its vertices. 
	In other words, the highest ranking level is the one that contains the vertex with the largest weight, the second highest ranking level is the one that contains the vertex with the second largest weight, and so on.
	Since every two vertices in the same level have the same weight, a complete ranking exists.

\begin{lemma}\label{maxWeightLem}
	Let $T_h$ be a complete binary tree, and let $L\subseteq V(T_h)$ be the set of leaves of $T_h$. 
	If $v_l$ denotes an arbitrary vertex at level $l$ in $T_h$, then for $5\leq i< h$, we have
\begin{equation*}
w_{L}(v_1) > w_{L}(v_2)> w_{L}(v_3) > w_{L}(v_0)>w_{L}(v_4) >\dots> w_{L}(v_i)>w_{L}(v_{i+1}) > \cdots.
\end{equation*}
\end{lemma}

\begin{proof}
	Let $v_l$ and $v_{l+1}$ be vertices at levels $l$ and $l+1$ of $T_h$ respectively. 
	Then by Lemma \ref{summedWeightLem},

\begin{align*}
w_{L}(v_l)-w_{L}(v_{l+1})
&= \frac{(2\omega)^{h-l}}{\omega^2}\left(1-\omega(2\omega^2)^l	\right)- \frac{(2\omega)^{h-l-1}}{\omega^2}\left(1-\omega(2\omega^2)^{l+1}	\right)\\
&= \frac{(2\omega)^{h-l-1}}{\omega^2}\left[2\omega-(2\omega^2)^{l+1}-1+	\omega(2\omega^2)^{l+1}\right]\\
&= \frac{(2\omega)^{h-l-1}}{\omega^2}\left[(2\omega^2)^{l+1}(\omega-1)+(\omega-1)+\omega\right]\\
&= \frac{(2\omega)^{h-l-1}}{\omega^2}\left[(\omega-1)\left((2\omega^2)^{l+1}+1\right)+\omega\right].
\end{align*}

	So $w_{L}(v_l)-w_{L}(v_{l+1})>0$ when $\omega>(1-\omega)\left[(2\omega^2)^{l+1}+1	\right]$, i.e. for $1\leq l\leq h-1$. 
	Therefore, the summed weight of each vertex decreases as the level of the vertex increases, except when the level increases from 0 to 1. 
	
	Now all we need to prove is that $w_{L}(v_4)<w_{L}(v_0) <w_{L}(v_3)$.
	This can be shown directly by calculating the ratios between the weights:

\begin{equation*}
\frac{w_{L}(v_0)}{w_{L}(v_4)}
=\frac{(2\omega)^{h}}{(2\omega)^{h-4}\left(1-\omega(2\omega^2)^4	\right)}
=\frac{(2\omega)^{4}}{1-16\omega^9	}
\approx 1.12937.
\end{equation*}

	Therefore, $w_{L}(v_4)<w_{L}(v_0)$. 
	Finally, we calculate the ratio of $w_{L}(v_0)$ to $w_{L}(v_3)$:
	
\begin{equation*}
\frac{w_{L}(v_0)}{w_{L}(v_3)}
=\frac{(2\omega)^{h}}{(2\omega)^{h-3}\left(1-\omega(2\omega^2)^3	\right)}
=\frac{(2\omega)^{3}}{1-8\omega^7	}
\approx 0.995713.
\end{equation*}

	Thus, $w_{L}(v_0)<w_{L}(v_3)$.
\end{proof}

We are now ready to present our bounds on the optimal-pegging number of $T_h$.

\begin{theorem}\label{optPegBinaryThm}
For a complete binary tree of height $h\geq 8$, $\frac{|V(T_h)|}{16}<p(T_h)<\frac{|V(T_h)|}{8}$.
\end{theorem}

\begin{proof}

	Let us first prove the lower bound. 
	Let $L$ be the set of leaves of $T_h$. 
	Since $h> 7$, by Lemma \ref{maxWeightLem} we know that the $h-3$ levels of vertices with the largest summed weights with respect to $L$ are levels 0 through $h-4$. 
	Let $v_l$ denote a vertex in level $l$. 
	Since each level $l$ has $2^l$ vertices, we can calculate the summed weight with respect to $L$ of the distribution $D$ made up of all the vertices in the first $h-3$ levels:

\begin{align*}
w_{L}(D)&= \sum\limits_{i=0}^{h-4}2^i w_{L}(v_i)\\
&= \sum\limits_{i=0}^{h-4}\frac{2^i(2\omega)^{h-i}}{\omega^2}\left(1-\omega(2\omega^2)^i	\right)\\
&= \sum\limits_{i=0}^{h-4}\frac{2^h\omega^{h-i}}{\omega^2}\left(1-\omega(2\omega^2)^i	\right)\\
&= \frac{2^h}{\omega^2}\sum\limits_{i=0}^{h-4}\left(\omega^{h-i}-2^i\omega^{h+i+1}	\right)\\
&= \frac{2^h}{\omega^2}\sum\limits_{i=0}^{h-4}\left(\omega^{h-i}-\omega^{h+1}(2\omega)^i	\right)\\
&= \frac{(2\omega)^h}{\omega^2}\sum\limits_{i=0}^{h-4}\left(\omega^{-i}-\omega(2\omega)^i	\right)\\
&= \frac{(2\omega)^h}{\omega^2}\left[\sum\limits_{i=0}^{h-4}\omega^{-i}-\omega \sum\limits_{i=0}^{h-4}(2\omega)^i\right]\\
&= \frac{(2\omega)^h}{\omega^2}\left[\frac{\omega^{3-h}-1}{\omega^{-1}-1}-\omega \left(\frac{(2\omega)^{h-3}-1} {2\omega -1}\right)\right]\\
&= \frac{(2\omega)^h}{\omega^2}\left[\frac{\omega^{3-h}-1}{\omega}-\frac{(2\omega)^{h-3}-1} {\omega^2}\right]\\
&=2^h(1-\omega^{h-3}+2^h\omega^{2h-7}-\omega^{h-4})
< 2^h
= |L|.
\end{align*}

	So $p(T_h)\geq \sum\limits_{i=0}^{h-4}2^i +1=  2^{h-3} >\frac{|V(T_h)|}{16}$.

	The upper bound can be established simply by putting pegs on every vertex of levels 0 through 5 of a tree of height 8. 
	It is possible to reach all the vertices, and since $T_8$ is a subtree of any complete binary tree of height greater than 8, this result can be generalized to all complete binary trees of height at least 8. 
	Furthermore, if we extend a path from the root of a complete binary tree of height 5 with pegs on every vertex of the tree, we can reach a distance of 4 along the path. 
	So given a tree of height 12, if levels 4 through 9 are covered with pegs, we can reach all the vertices in the tree. 
	Therefore, $\frac{(2^4+\dots +2^9)|V(T_h)|}{2^{13}-1}=\frac{1008|V(T_h)|}{8191}$, and so $p(T_h)<\frac{1008|V(T_h)|}{8191}<\frac{|V(T_h)|}{8}$.
\end{proof}

\section{Infinitary Trees}\label{infTrees}

	A \textit{complete infinitary tree of height $h$} is a tree rooted at a vertex $r$ such that each vertex has a countably infinite number of children and the distance between $r$ and any leaf is $h$. 
	Since each level consists of an infinite number of pairwise non-adjacent vertices, it is impossible to calculate the pegging number of such a tree in any meaningful way. 

	However, finding specific finite distributions of pegs that can peg any vertex in the tree is within our reach. 
	In this section we compute the optimal-pegging number of the complete infinitary tree of height $h$.

\begin{lemma}\label{infTreeLem}
	Let $T_h$ denote a complete infinitary tree of height $h$ rooted at vertex $r$, and let $F_n$ denote the $n$th Fibonacci number. 
	Then any distribution $D_k$ which pegs levels 0 through $k$ requires at least $F_{k+3}-1$ pegs.
	Moreover, $r\in D_k$.
\end{lemma}

\begin{proof}We prove this by strong induction on $k$. Since the base cases are trivial, we assume that the theorem statement is true for all $2\leq i<k$, where $k\leq h$. Let $D_k$ be a minimal distribution needed to reach levels 0 through $k$ of $T_h$. 
	Since $D_k$ has finitely many pegs, by the Pigeonhole Principle, there exists at least one subtree of $T_h$ rooted at a child of $r$ with no vertices in $D_k$. 

	Let $w$ be a vertex at level $k$ in this subtree of $T_h$, and let $M$ be a minimum finite sequence of pegging moves used to reach $w$ from $D_k$. 
	The last move $m\in M$ is $m=\peg{u}{v}{w}$ for some two vertices $u$ in level $k-2$ and $v$ in level $k-1$. 
	By Lemma \ref{ancestorLem}, we can partition $D_k$ into two distributions, one that can peg $u$ and one that can peg $v$.
	Suppose $r\notin D_k$. 
	Then neither of the parts contain $r$, and by the inductive hypothesis, they each contain at least $F_{k+2}$ and $F_{k+1}$ vertices, respectively. 
	This implies that $|D_k|\geq F_{k+2}+F_{k+1}>F_{k+3}-1$. 
	So in order for $D_k$ to be of size $F_{h+3}-1$, it must contain $r$.




 
	However, since both partitions cannot contain $r$, one must be sub-optimal. 
	It is evident that if we can place a peg on a vertex in the $k$th level of this subtree, then we are able reach all the vertices in the first $k-1$ levels of the subtree.
	Then

\begin{equation*}
|D_k|
\geq (F_{k+1} -1) + (F_{k+2}-1)+1
= F_{k+3}-1.
\end{equation*}

	The distributions for $D_0$ and $D_1$ are obvious.
	Now we will show how to recursively construct a distribution of size $F_{k+3}-1$ which pegs level 0 through $k$ for a tree of height $h\geq k$ where $k\geq2$.
 
	Let $T_h$ be an infinitary tree of height $h$ and let $D_k=D_{k-2}\cup D_{k-1} \cup\{x,y\}$ where $D_{k-1}$ and $D_{k-2}$ are optimal distributions for pegging the first $k-1$ level and the first $k-2$ level of $T_{h}$, respectively, $x$ and $y$ are not contained in $D_{k-1}$ or $D_{k-2}$, and $x$ is adjacent to both $y$ and $r$. 
	Since $T_h$ is an infinitary tree, we can choose $D_{k-1}$ and $D_{k-2}$ so that they intersect only at $r$.
	We can use $D_{k-1}$ to reach any vertex $v$ a distance of $k-1$ away from $r$, and $D_{k-2}\cup\{x,y\}\backslash\{r\}$ to reach any vertex a distance of $k-2$ away from $r$. 
	If we choose to use $D_{k-2}\cup\{x,y\}$ to peg $u$, the vertex a distance of $k-2$ away from $r$ and adjacent to $v$, then we can perform the pegging move $m=\peg{u}{v}{w}$ where $w$ is a vertex adjacent to $v$ that is a distance of $k$ away from $r$. 
	So we can reach any vertex in levels 0 through $k$ of $T_h$.
	The size of this distribution is
	
\begin{equation*}
|D_k| =|D_{k-1}|+|D_{k-2}|-1+2
= (F_{k+1} -1) + (F_{k+2}-1)+1
= F_{k+3}-1.
\end{equation*}
\end{proof}
\begin{theorem}\label{infTreeThm}
Let $T_h$ be an infinitary tree of height $h$. Then
\begin{equation*}
p(T_h) = F_{h+3}-1
\end{equation*}
where $F_n$ denotes the $n$th Fibonacci number.
\end{theorem}

\begin{figure}[h]
	\centering
		\includegraphics[scale=0.4]{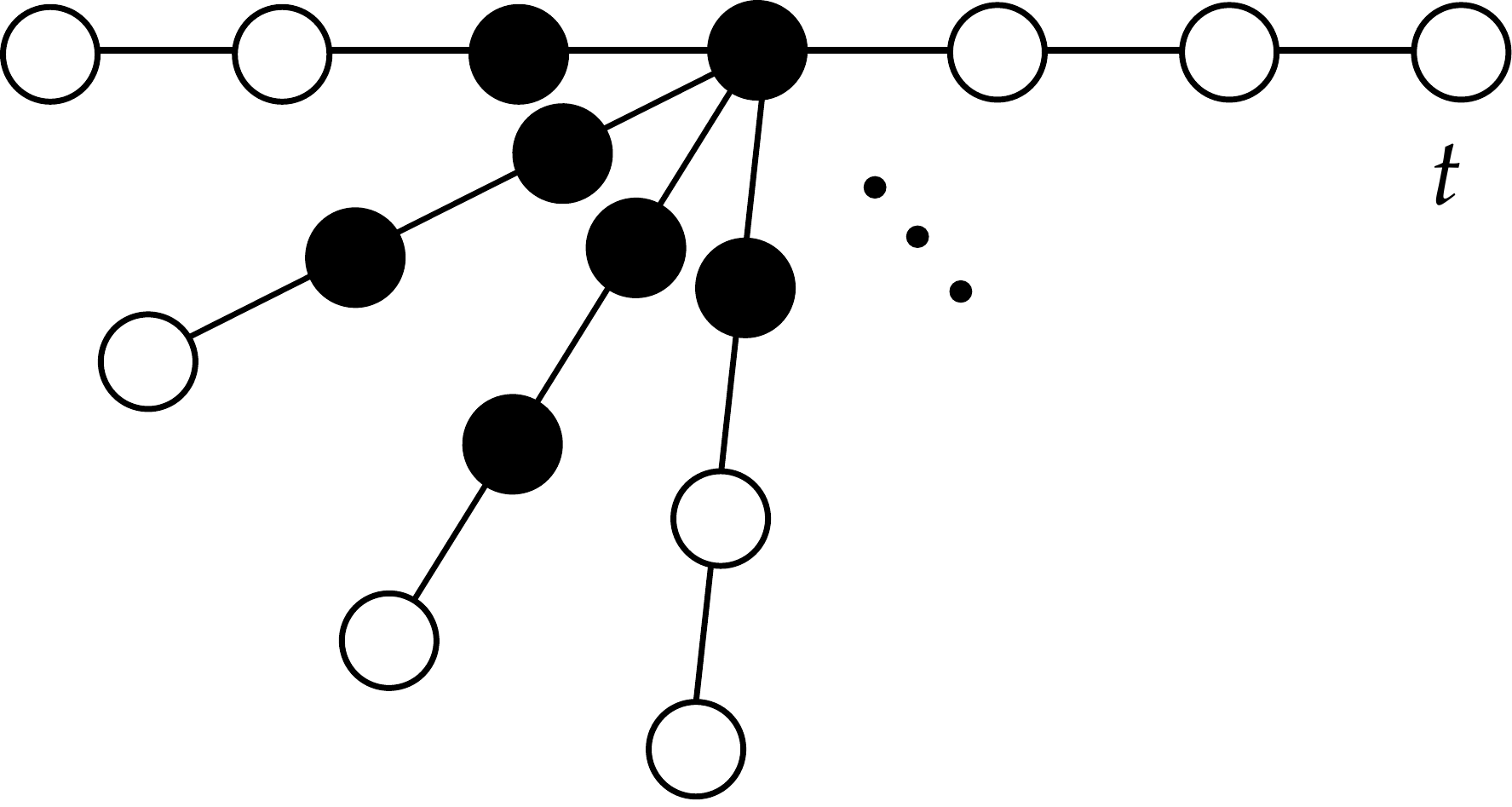}
	\caption{The infinitary tree $T_3$ with distribution $D_3$ containing $F_6-1$ pegs that can reach target vertex $t$.}
\end{figure}

\begin{proof}
	By Lemma \ref{infTreeLem}, we need at least $F_{h+3} - 1$ pegs to peg levels 0 through $h$ of $T_h$, so $p(T_h)\geq F_{h+3}-1$. 
	Also by Lemma \ref{infTreeLem}, there exists a distribution $D$ of this size such that $R(D)=V(T_h)$. 
	So $p(T_h)=F_{h+3}-1$.
\end{proof}

	By combining our results from Theorems \ref{optPegBinaryThm} and \ref{infTreeThm}, we are confronted by the surprising and counterintuitive reality of the existence of trees which have optimal-pegging numbers that increase with the removal of a leaf. 
	We show this in the following corollary.

\begin{cor}\label{supertrees}
There exist infinitely many trees which have optimal-pegging numbers that increase with the removal of a leaf.
\end{cor}

\begin{proof}
	Let $T_h$ and $T'_h$ denote a complete binary tree and an infinitary tree, respectively, both of which have height $h$. 
	By Theorems \ref{optPegBinaryThm} and \ref{infTreeThm}, $p(T_h)> p(T'_h)$ provided that $\frac{1}{16}|V(T_h)|> F_{h+3}-1$, or

$$
\frac{1}{16}(2^{h+1}-1)>\frac{1}{\sqrt{5}}\left(\frac{1+\sqrt{5}}{2}	\right)^{h+3}-1.
$$

This is true if

\begin{align*}
\frac{2^h}{8}&\geq\frac{1}{\sqrt{5}}\left(\frac{1+\sqrt{5}}{2}\right)^{h+3}\\
&=\frac{1}{\sqrt{5}}\left(\frac{1+\sqrt{5}}{2}\right)^3\left(\frac{1+\sqrt{5}}{2}\right)^{h}\\
&=\frac{2^h(1+\sqrt{5})^3}{8\sqrt{5}}\left(\frac{1+\sqrt{5}}{4}\right)^{h}\\
\end{align*}

or equivalently

$$
\frac{\sqrt{5}}{(1+\sqrt{5})^3}\geq\left(\frac{1+\sqrt{5}}{4}\right)^h.
$$

	This is true for $h=13$. Since $(1+\sqrt{5})/4<1$, $((1+\sqrt{5})/4)^h$ is a monotonically decreasing function. 
	Hence $p(T_h)>p(T'_h)$ for all $h\geq 13$.

	Let $D$ be the recursively defined distribution on the infinitary tree $T'_h$ described in the proof of Lemma \ref{infTreeLem}.
	By Lemma \ref{infTreeLem}, the root of $T'_h$ is covered with a peg, and the distance between every vertex in $D$ and the root is at most 2. 
	Therefore, there exists a complete $(F_{h+3}-2)$-ary tree of height $h$, $T''_h$, such that $T''_h$ is a subtree of $T'_h$, and $D\subseteq V(T''_h)$. 
	By Lemma \ref{Melanie}, there exists a finite sequence of pegging moves $M$ such that $\anc(D,M,v)\subseteq V(T''_h)$ for every $v\in T''_h$. 
	Therefore, $p(T''_h)\leq p(T'_h)<p(T_h)$ for all $h\geq 13$. 
	Since $T''_h$ is finite, by the Pigeonhole Principle, there exists at least one tree $T$ that is a subtree of $T''_h$ and a supertree of $T_h$ such that the deletion of a leaf from $T$ results in a tree with a larger optimal-pegging number.
\end{proof}

\section{Caterpillars and Lobsters}\label{caterpillars}

	This section is devoted to computations of the optimal-pegging numbers of caterpillar graphs, as well as calculating the lowest upper bound on the optimal-pegging number of lobster graphs. 
	A \textit{caterpillar graph}, $E$, is a tree such that if all leaves and their incident edges are removed, the remainder of the graph forms a path. 
	A \textit{lobster graph}, $L$, is a tree such that if all leaves and their incident edges are removed, the remainder of the graph is a caterpillar.

\begin{theorem}\label{caterpillarThm}
The caterpillar graph $E$ satisfies $p(E)=\left\lceil \frac{d+1}{2}\right\rceil$, where $d$ is the diameter of $E$.
\end{theorem}

First we need four lemmas that will be used in the proof of this theorem.

\begin{lemma}\label{adjLem}
	Let $P_n$ be a path on $n$ vertices and let $D$ be a distribution of pegs on $P_n$. 
	Then for all $v\in R(D)$, $v\in D$ or $v$ is adjacent to some vertex in $D$.
\end{lemma}

\begin{proof}
	Suppose that this claim is false, and that there exists a vertex $v \in R(D)$ that is a distance of at least two from any vertex in $D$. 
	If $v\in R(D)$, then it must be reachable from the distribution of pegs on one side of it.  
	Let $D'\subseteq D$ be the distribution of pegs on one side of $v$. 
	We can get an upper bound for the weight of $D'$:
	
\begin{equation*}
w_v(D')
\leq \omega^2+\omega^3+\cdots+\omega^{n-1}
= \omega^2\left(\frac{1-\omega^{n-2}}{1-\omega}\right)
< \frac{\omega^2}{1-\omega}
= 1.
\end{equation*}

	Therefore $v\notin R(D')$, and so $v\notin R(D)$.
\end{proof} 

\begin{lemma}\label{caterpillarLem1}
	Let $D$ be a distribution of pegs on some caterpillar $E$ with diameter $d\geq2$, and let $P$ be a longest path of $E$. 
	If $P\subseteq R(D)$, then $R(D)=V(E)$.
\end{lemma}

\begin{proof}
	By definition of a caterpillar, a vertex that does not lie in $P$ must be a leaf. 
	Let $l$ be such a leaf, and let $v\in P$ be adjacent to $l$. 
	Assume for the sake of contradiction that $l\notin R(D)$. 
	Since $v$ is not a leaf there exists two vertices $u$ and $w$ adjacent to $v$ that are in $P$, and are therefore in the reach of $D$. 
	Suppose there exists a sequence of moves $M=(m_1,\dots,m_k)$ where $v\in\anc(D,M,u)$. 
	Then there must be some $x\in R(D)$ (possibly $w$) adjacent to $v$ such that either $x\in D$ or $v,x\in M_i(D)$, for some $1\leq i\leq k$. 
	Either way, we can reach $l$ via the pegging move $m=\peg{x}{v}{l}$. 
	The same argument holds if there exists a sequence of moves $M$ where $v\in\anc(D,M,w)$. 
	Now suppose this is not the case, and that no such sequence of moves exists where $v$ is an ancestor of $u$ or $w$. 
	Let us consider the case where there exists a sequence of moves $M=(m_1,\dots,m_k)$ where $u\in\anc(D,M,v)$. 
	If there is some finite sequence of pegging moves $M'=(m'_1,\dots,m'_n)$ such that $u\in\anc(D,M',w)$, since $v$ lies in the path between $u$ and $w$, and $v$ is adjacent to $w$, then $m'_i=\peg{x}{v}{w}$ for some $1\leq i\leq n$ and $x\in V(E)$. 
	So we can construct a sequence of pegging moves $M''=(m''_1,\dots,m''_n)$ where $m''_j = \peg{x}{v}{l}$ if $j=i$ and $m''_j=m'_j$ otherwise, and we can reach $l$. 
	If such a sequence $M'$ does not exist, then by Lemma \ref{ancestorLem}, there exists another sequence $M''=(m''_1,\dots,m''_n)$ such that $\anc(D,M,v)$ and $\anc(D,M'',w)$ are disjoint. 
	Let $\left<M,M''\right>=(m_1,\dots,m_k,m''_1,\dots,m''_n)$. 
	Then $v,w\in\left<M,M''\right>(D)$, and so we can reach $l$ via the pegging move $m=\peg{w}{v}{l}$. 
	If $u$ is not an ancestor of $v$ and vice versa, we can reach $l$ via the pegging move $\peg{u}{v}{l}$. 
	So $l\in R(D)$.
\end{proof}
\begin{lemma}\label{caterpillarLem2}
	Let $E$ be a caterpillar with  diameter $d\geq 3$, and let $D$ be a distribution of pegs on $E$ such that $R(D)=V(E)$. 
	Denote the set of leaves in $D$ as $L_D$. 
	If $|L_D|\geq1$, then there exists a distribution $D'$ such that $|D'|\leq|D|$, $R(D')= R(D)$, and $|L_{D'}|<|L_D|$.
\end{lemma}

\begin{proof}Let $l$ be a leaf in $D$ and $v$ be the vertex adjacent to $l$. There are three possible configurations of $D$:

	\textit{Case 1:} $v\notin D$. 
	Since $l$ is not adjacent to any other pegged vertices, it cannot be used to reach $v$. 
	So $v\in R(D\backslash\{l\})$. 
	If $D'=D\cup\{v\}\backslash\{l\}$, then after a series of stacking moves we can achieve 2 pegs on $v$ with which to perform a pebbling move to reach $l$. 
	But by the Peggling Lemma, this means that $\{v,l\} \subseteq R(D')$. 
	Furthermore, the set of vertices that can be reached by a pegging move using $l$ are all adjacent to $v$, so they can be reached with pebbling moves as well. 
	So $R(D')=R(D)$. 


	\textit{Case 2:} $v\in D$. 
	Suppose there exists a path $P\subseteq V(E)$ of length $d$ such that $P\subseteq D$. 
	If  $l\notin P$, then by Lemma \ref{caterpillarLem1}, $R(D\backslash\{l\})= R(D)$. 
	Suppose $l\in P$. 
	Since $d\geq 2$, there is a vertex $u\in P$ that is distinct from $l$ and adjacent to $v$, and we can perform the stacking move $m=\peg{u}{v}{l}$. 
	So $P\subseteq R(P\backslash\{l\})$, and by Lemma \ref{caterpillarLem1}, $R(D\backslash\{l\})= R(D)$. 

	Suppose no such path exists. 
	Denote the two closest unpegged, non-leaf vertices on opposite sides of $v$ as $u$ and $w$ (in other words, $v$ lies on the path between $u$ and $w$, and there is no unpegged, non-leaf vertex $v'$ such that $v'$ lies on the path between $v$ and $u$ or $v$ and $w$). Let $P$ denote the longest path of $E$ containing $u$, $v$, and $w$. 
	By the definitions of $u$ and $w$, any other unpegged vertex $x$ in $P$ is closer to either $u$ or $w$ than it is to $v$. 
	Suppose $x$ is closer to $u$. By Lemma \ref{Melanie}, if $l$ is used to reach $x$, we move the peg on $l$ toward $x$, and therefore toward $u$. 
	So if $M=(m_1,\dots,m_k)$ is the sequence of moves on $D$ used to reach $x$ by first moving the peg on $l$, and if $i$ is the smallest index such that $M_i$ contains $u$, then $x\in R(M_i(D))$. 
	Since the pegs on $l$ and $v$ are removed from $E$ in the first $i$ moves of $M$, $l,v\notin M_i(D)$. 
	Let $D'=D\cup\{u,w\} \backslash\{l,v\}$. 
	Then $M_i(D)\subseteq D'$, and so $x\in R(D')$. 
	We can follow a similar argument if $x$ is closer to $w$. 
	So $P\subseteq R(D')$, and by Lemma \ref{caterpillarLem1}, $R(D')=V(E)$.
\end{proof}


The following lemma and its proof were presented in Helleloid, et al. \cite{Moulton}.

\begin{lemma}\label{pathLem} 
For $n\geq 3$, the path $P_n$ satisfies $p(P_n)=\lceil n/2\rceil$.
\end{lemma}

\begin{proof}[Proof of Theorem \ref{caterpillarThm}]
	By Lemmas \ref{caterpillarLem1} and \ref{pathLem}, we see that $p(E)\leq \left\lceil \frac{d+1}{2}\right\rceil$ and by Lemma \ref{caterpillarLem2}, it is clear that $p(E)\leq p(P_{d+1})$. 
	So $p(E)= \left\lceil \frac{d+1}{2}\right\rceil$.
\end{proof}

	The optimal-pegging numbers for lobsters are not as straightforward. 
	A caterpillar is a type of lobster, so some lobsters with diameters $d$ have optimal-pegging numbers as low as $\left\lceil \frac{d+1}{2}\right\rceil$. 
	However, this is not the case for all lobsters. 
	A counterexample is a lobster $L$ such that if all its leaves and their incident edges are removed, the remainder of the graph is a star with at least 4 leaves. 
	The diameter of this graph is 4, and so for the sake of contradiction, suppose we only need a distribution $D$, of 3 pegs to reach all the vertices in $L$. 
	By the Pigeonhole Principle, there is at least one leaf $l\notin D$ such that $d(l,v)\geq 2$ where $v$ is the closest vertex to $l$ in $D$, and no other vertex in $D$ is as close to $l$ as $v$ is. 
	The closest we can get a peg to $l$ after one move is a distance 1 away from $l$. 
	However, since the third peg is a distance of at least 3 away from $l$, the two pegs remaining on the graph are not adjacent, and so no more pegging moves can be performed. 
	Therefore, 3 pegs are not sufficient to peg $L$.

\begin{figure}[h]
	\centering
		\includegraphics[scale=0.3]{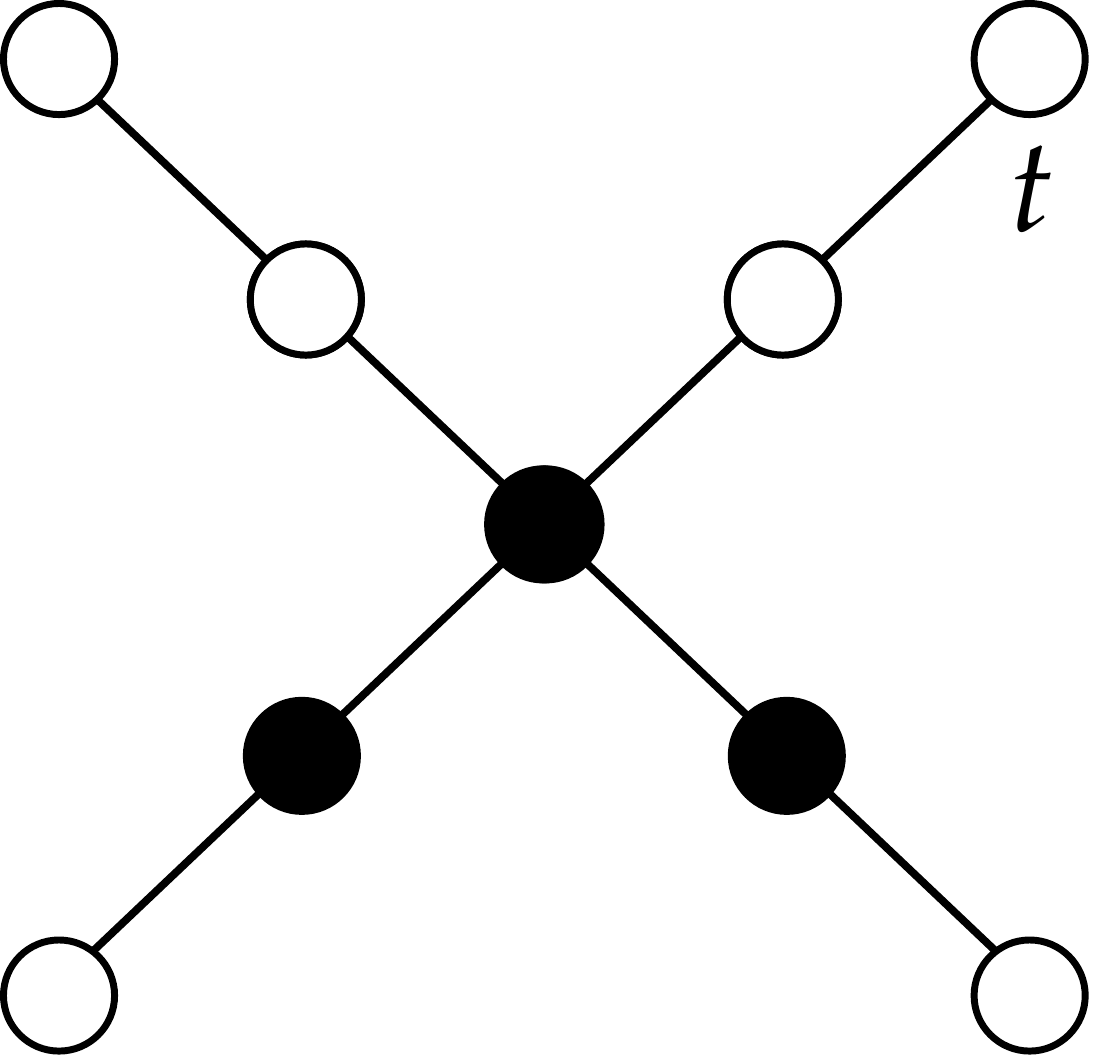}
	\caption{An example of a lobster $L$ where $p(L)>3$. This graph shows a distribution of three pegs such that $t$ cannot be reached.}
\end{figure}
	For this reason, we present only an upper bound for the optimal-pegging number of a lobster.

\begin{theorem}\label{lobsterThm}
For all lobsters $L$ with diameter $d\geq 5$, we have $p(L)\leq d-1$.
\end{theorem}

\begin{proof}
	Consider the distribution $D$ where all the pegs cover a longest path of $L$ with the exception of the two endpoints. 
	Clearly, if we can reach every leaf in $L$, then we can reach all the vertices in $L$. 
	Since $|D|=d-1\geq 2$, any vertex in, or adjacent to, some vertex in $D$ is reachable, so let us consider the case where the leaf in question is a distance of at least two away from any vertex in $D$. 
	Let $u,v,l\in V(L)$ where $l$ is the leaf, $v\in D$, and $u$ is adjacent to both $v$ and $l$. 
	Since $u,l\notin D$ we know that there are two vertices $x$ and $y$ adjacent to $v$ that are covered with pegs. 
	Furthermore, since $D\geq 4$, there must exist another vertex $z$ adjacent to either $x$ or $y$ in $D$. 
	Without loss of generality, assume that $z$ is adjacent to $y$. 
	Then we can reach $l$ via a series of pegging moves:
	
\begin{align*}
m_1 &= \peg{x}{v}{u}\\
m_2 &= \peg{z}{y}{v}\\
m_3 &= \peg{v}{u}{l}.
\end{align*}

	So $l\in R(D)$.
\end{proof}

\section{Future Research and Acknowledgments}\label{conclusion}


	Future research includes finding bounds on complete $n$-ary trees for $n\geq3$. 
	We can also study various types of caterpillars and lobsters in greater depth. 
	By classifying these graphs by properties such as the number of their leaves, their diameters, or the regularity of their constructions, one might be better able to compute their pegging numbers. 
	Another direction would be to find the probabilities that specific graphs can be pegged based on the size of their distributions with respect their number of the vertices.

	The research for this paper was conducted at the University of Minnesota Duluth and was made possible by the financial support of the National Security Agency (grant number DMS-0447070-001) and the National Science Foundation (grant number DMS-0447070-001). 
	I would like to extend special thanks to Ricky Liu, Reid Barton, and David Moulton for the ideas, suggestions, and intellectual support that they provided, and to Joe Gallian for proposing the topic of this paper and for running the REU at the University of Minnesota Duluth. 
	I would also like to thank David Arthur, Mike Develin, Stephen Hartke, Geir Helleloid, Philip Matchett-Wood, Joy Morris, Julie Savitt and Melanie Wood for their input and suggestions.


\end{document}